\documentclass[11pt,hyp,letter]{amsart}
\usepackage{hyperref}
\hypersetup{nesting=true,debug=true,naturalnames=true}
\usepackage{graphicx,amssymb,upref}

\usepackage{amsmath,amsthm}
\usepackage{amsfonts}
\usepackage{latexsym}
\usepackage{amsbsy}
\usepackage{amssymb,mathscinet}

 \newtheorem{theorem}{Theorem}[section]
 \newtheorem{corollary}[theorem]{Corollary}
 \newtheorem{lemma}[theorem]{Lemma}
 \newtheorem{proposition}[theorem]{Proposition} 

 \theoremstyle{definition}
 \newtheorem{definition}[theorem]{Definition}
 \newtheorem{remark}[theorem]{Remark}
 \newtheorem{notation}[theorem]{Notation}
 
 \theoremstyle{remark}

 \numberwithin{equation}{section}

\newcommand{\be}{\mathbb E}
\newcommand{\bn}{\mathbb N}
\newcommand{\Nk}{\bn_0^k}
\newcommand{\ot}{\otimes}
\newcommand {\id} {{\textrm{id}}}
\newcommand{\wt}{\widetilde}
\newcommand{\wT}{\wt{T}}
\newcommand{\onek}{\{1,\ldots,k\}}

\begin{document}

\title[Generating wandering subsp. for doubly comm. cov. repr.]
 {Generating wandering subspaces for doubly commuting covariant representations}

\date{\today}






\author[Trivedi]{Harsh Trivedi\textsuperscript{*}}

\address{The LNM Institute of Information Technology, Rupa ki Nangal, Post-Sumel, Via-Jamdoli
Jaipur-302031,
(Rajasthan) INDIA}
\email{harsh.trivedi@lnmiit.ac.in, trivediharsh26@gmail.com}

\author[Veerabathiran]{Shankar Veerabathiran}
\address{Ramanujan Institute for Advanced Study in Mathematics,
University of Madras, Chennai (Madras) 600005, India}
\email{shankarunom@gmail.com}
\thanks{*corresponding author}


\begin{abstract}

We obtain a Halmos-Richter-type wandering subspace theorem for covariant representations of $C^*$-correspondences.
Further the notion of Cauchy dual and a version of Shimorin's Wold-type decomposition for covariant representations of $C^*$-correspondences is explored and as an application a wandering subspace theorem for doubly commuting covariant representations is derived. Using this wandering subspace theorem generating wandering subspaces are characterized for covariant representations of product systems in terms of the doubly commutativity condition.
\end{abstract}

\keywords{Hilbert $C^*$-modules, isometry, covariant representations,
product systems, tuples of operators, doubly commuting, Nica covariance,  Shimorin property,
wandering subspaces, Wold decomposition, Fock space}
\subjclass[2010]{46L08, 47A13, 47A15, 47B38, 47L30, 47L55, 47L80.}

\maketitle
\section{Introduction}

Beurling's theorem \cite{B49} says that if $\mathcal K$ is a closed shift invariant subspace of the Hardy space $H^2(\mathbb D)$ then $\mathcal K$ is the image of an inner function. Wold \cite{W} gave a decomposition of isometries: ``Every isometry on a Hilbert space is either a shift, or a unitary, or decomposes uniquely as a direct sum of a shift and a unitary."  In fact, the Beurling's theorem is a corollary to the Wold decomposition (cf. \cite{D11}).

In \cite{R69}, Rudin explained that the Beurling's theorem fails in general in the multivariable case. However, using the Wold-type decomposition due to S{\l}oci{\'n}ski \cite{Sl80} for a pair of doubly commuting isometries, Mandrekar \cite{M88} proved a version of the Beurling's theorem for the Hardy space over the bidisc $H^2(\mathbb D^2)$. Sarkar, Sasane and Wick \cite{S13} proved a Beurling-type theorem in the polydisc case and their analysis is based on a generalization of the result of S{\l}oci{\'n}ski given in \cite{S14}. In \cite{ARS96}, Aleman, Richter and Sundberg obtained a Beurling-type theorem for the Bergman space $A^2(\mathbb D).$ Shimorin \cite{S01} provided an elementary proof of this result by considering a Wold-type decomposition for operators close to isometries. Shimorin's result is generalized in \cite{RT10} for the Bergman space $A^2(\mathbb D^2)$ and in \cite{CDSS14} for the Bergman space $A^2(\mathbb D^n).$

In recent years, there has been several papers on the Wold decomposition in the multivariable Operator Theory, see \cite{MS99,Po89,S14,SZ08,Sl80,V11,TV19}. Pimsner \cite{P97} extended the notion of Cuntz algebras \cite{C77} and showed that a covariant representation of a $C^*$-correspondence $E$ extends to the Toeplitz algebra of $E$ if and only if the covariant representation is isometric. Muhly and Solel \cite{MS99} explored the Wold decomposition, for $C^*$-representations of tensor algebras of $C^*$-correspondences, and also analyzed the invariant subspace structure of certain
subalgebras of Cuntz-Krieger algebras. The terminology of the tensor product system of Hilbert spaces has been used by Arveson \cite{A89} to classify $E_0$-semigroups. Discrete product system of $C^*$-correspondences is studied by Fowler \cite{F02}. Solel \cite{S08} defined the notion of doubly commuting covariant representations of product systems of $C^*$-correspondences and explored their regular dilations.
Wold decomposition for the doubly commuting isometric covariant representations, which is a higher rank version of the S{\l}oci{\'n}ski decomposition, is due to Skalski-Zacharias \cite{SZ08}.

One of the well-known implication of the Wold decomposition is that the wandering subspace of the shift is the kernel of its adjoint.  Halmos \cite{H61} characterized all the shift invariant subspaces by proving a wandering subspace theorem which states that there is a one-to-one correspondence between the set of all wandering subspaces and the set of all invariant subspaces. Richter \cite{R88} obtained a wandering subspace theorem for the general case of analytic operators satisfying certain condition. In Section \ref{Sec2}, a version of Richter's wandering subspace theorem for covariant representations of $C^*$-correspondences is proved. Based on the computations of Section \ref{Sec2}, in Section \ref{Sec3} we explore the notion of Cauchy dual and an analogue of Shimorin's Wold-type decomposition. Using the results of Section \ref{Sec3}, in the last section we generalize the main theorem of \cite{CDSS14}, which is a wandering subspace theorem for doubly commuting bounded operators, for the doubly commuting covariant representations.

\subsection{Preliminaries and Notations}
We assume the elementary theory of Hilbert $C^*$-modules (cf. \cite{MR0355613, La95}). Moreover, in this subsection we survey elementary definitions and properties related to covariant representations of
$C^*$-correspondences (cf. \cite{P97,MR1648483}) and recall the Wold-type decomposition for isometric covariant representations from \cite{H16, MS99}.

Assume $E$ to be a Hilbert
$C^*$-module over a $C^*$-algebra $\mathcal M.$ Let $\mathcal L(E)$ denote the $C^*$-algebra of all
adjointable operators on $E$. We call $E$ a {\it
$C^*$-correspondence over $\mathcal M$} if $E$ has a left $\mathcal
M$-module structure defined by a non-zero $*$-homomorphism
$\phi:\mathcal M\to \mathcal L(E)$ in the following sense
\[
a\xi:=\phi(a)\xi \quad \quad (a\in\mathcal M, \xi\in E).
\]
Every $*$-homomorphism considered in this paper is
essential, that is, the closed linear span of
$\phi(\mathcal M)E$ is $E.$ Every $C^*$-correspondence has usual operator space structure induced from viewing it as a corner in corresponding linking algebra. If $F$ is a
$C^*$-correspondence over $\mathcal M,$ then we get the notion of
tensor product $F\otimes_{\phi} E$  which has the properties
\[
(\zeta_1 a)\otimes \xi_1=\zeta_1\otimes \phi(a)\xi_1,
\]
\[
\langle\zeta_1\otimes\xi_1,\zeta_2\otimes\xi_2\rangle=\langle\xi_1,\phi(\langle\zeta_1,\zeta_2\rangle)\xi_2\rangle
\]
for every $\zeta_1,\zeta_2\in F;$ $\xi_1,\xi_2\in E$ and $a\in\mathcal
M.$  

Unless it is necessary, the next section onwards we denote the tensor product $F\otimes_{\phi} E$  simply by  $F\otimes E.$  
\begin{definition}
Assume $\mathcal H$ to be a Hilbert
space, and $E$ to be a $C^*$-correspondence over a $C^*$-algebra $\mathcal M$. Consider a representation
$\sigma:\mathcal M\to B(\mathcal H)$ and a linear map $T:
E\to B(\mathcal H).$ Then the tuple $(\sigma, T)$ is
called a {\it covariant representation} of $E$ on $\mathcal H$ (cf. \cite{MR1648483}) if
\[
T(a\xi b)=\sigma(a)T(\xi)\sigma(b) \quad \quad (\xi\in E;
a,b\in\mathcal M).
\]
We call the covariant representation $(\sigma, T)$ a {\it completely bounded} ( resp. {\it completely
contractive}) if $T$ is completely bounded (resp. completely contractive). Further, it is called {\it isometric} or {\it Toeplitz} if
\[
T(\xi)^*T(\zeta)=\sigma(\langle \xi,\zeta\rangle) \quad \quad
(\xi,\zeta\in E).
\]
\end{definition}

\begin{lemma}{\rm (\cite[Lemma
3.5]{MR1648483})}\label{MSL}
The map $(\sigma, T)\mapsto \widetilde T$ gives a bijection
between the collection of all completely bounded (resp.completely contractive), covariant
representations $(\sigma, T)$ of $E$ on $\mathcal H$ and the
collection of all bounded (resp. contractive) linear maps
$\widetilde{T}:~\mbox{$E\otimes_{\sigma} \mathcal H\to \mathcal
H$}$ defined by
\[
\widetilde{T}(\xi\otimes h):=T(\xi)h \quad \quad (\xi\in E,
h\in\mathcal H),
\]
and satisfying $\widetilde{T}(\phi(a)\otimes I_{\mathcal
H})=\sigma(a)\widetilde{T}$, $a\in\mathcal M$. Moreover, $\widetilde
T$ is an isometry if and only if the covariant representation $(\sigma, T)$ is isometric.
\end{lemma}
We say that the covariant representation $(\sigma, T)$ is {\it fully co-isometric} if $\widetilde{T}$ is co-isometric, that is, $\widetilde{T}\widetilde{T}^*=I_{\mathcal{H}}.$

Consider a $C^*$-correspondence $E$ over a $C^*$-algebra $\mathcal M.$ Then for every $n\in \mathbb{N}$,
$E^{\otimes n}: =E\otimes_{\phi} \cdots \otimes_{\phi}E$ ($n$ fold tensor product)  is  the $C^*$-correspondence over the  $C^*$-algebra $\mathcal M$, where the left action  of $\mathcal{M}$ on $E^{\otimes n}$  is defined by  $$\phi^n(a)(\xi_1 \otimes \cdots \otimes \xi_n):=\phi(a)\xi_1\otimes \cdots \otimes\xi_n.$$   We use $\mathbb N_0 :=\mathbb N\cup\{0\}$ and $E^{\otimes 0}:=\mathcal{M}$.

For each $ n \in \mathbb{N}$ define the map $\wT _n: E^{\ot n}\otimes_{\sigma} \mathcal H\to \mathcal H$ by
\[ \wT_n (\xi_1 \ot \cdots \ot \xi_n \ot h) = T (\xi_1) \cdots T(\xi_n) h\]
 where $\xi_1, \ldots, \xi_n\in E, h \in \mathcal H$. We get,
\begin{equation}\label{eqnn}\wT_n=\wT(I_{E} \otimes \wT) \cdots (I_{E^{\otimes n-1}} \otimes  \wT ). \end{equation}

The {\it Fock module} $\mathcal{F}(E):= \bigoplus_{n \in \mathbb N_0}E^{\otimes n}$ is the  $C^*$-correspondence over a $C^*$-algebra  $\mathcal M$, where the left action  of $\mathcal{M}$ on $\mathcal{F}(E)$ is  denoted by $\phi_{\infty}: \mathcal{M} \longrightarrow L(\mathcal{F}(E))$  and it is defined by $$\phi_{\infty}(a)(\oplus_{n \in \mathbb N_0}\xi_n):=\oplus_{n \in \mathbb N_0}\phi_{n}(a)\xi_n , \:\: \xi_n \in E^{\otimes n}.$$
For $\xi \in E$, we define the {\it creation operator} $T_{\xi}$ on $\mathcal{F}(E)$ by $$T_{\xi}(\eta):=\xi \otimes \eta, \:\: \eta \in E^{\otimes n}, n \in \mathbb N_0.$$

The following definition of induced representation is a generalization of the multiplication operator $M_z\otimes I_{\mathcal{H}}$ on the vector valued Hardy spaces $H^2(\mathbb D)\otimes \mathcal{H}$:
\begin{definition}
Let $\pi $ denote a representation of a $C^*$-algebra $\mathcal{M}$ on a Hilbert space $\mathcal{H}$. Let $E$ be a $C^*$-correspondence over $\mathcal M.$ The isometric covariant representation $(\rho, S)$ of $E$ on the Hilbert space $\mathcal{F}(E)\otimes_{\pi}\mathcal{H}$ defined as
\begin{align*}
\rho(a):&=\phi_{\infty}(a) \otimes I_{\mathcal{H}} \:\:, a \in \mathcal{M}&\\
S(\xi):&=T_{\xi}\otimes I_{\mathcal{H}} ,\:\: \xi \in E.
\end{align*}
is called an {\it induced representation} (cf. \cite{R74}) (induced  by $\pi$).
\end{definition}

\begin{definition}	
	\begin{itemize}
\item[(i)] Let $(\sigma, T)$ be a completely bounded, covariant representation of $E$ on a Hilbert space $\mathcal{H},$ and let  $\mathcal K$ be a closed subspace of $ \mathcal H.$  The subspace $\mathcal K$ is called $(\sigma, T)$-{\it invariant}(resp. $(\sigma, T)$-{\it reducing}) (cf. \cite{SZ08}) if it   is $\sigma(\mathcal M)$-invariant (i.e., the projection onto $\mathcal K$, will be denoted throughout by $P_{\mathcal K}$, lies  in
	$\sigma(\mathcal M)')$, and if $\mathcal K$ (resp. both $ \mathcal{K}, \mathcal{K}^{\bot}$) is  left invariant by each operator $T (\xi)$ for $\xi \in E.$ Then  the natural restriction of this representation provides a new representation of $E$ on $\mathcal{K}$ and it will denoted by $(\sigma, T)|_{\mathcal{K}}.$

\item[(ii)]  For a closed $\sigma(\mathcal{M})$-invariant subspace $\mathcal{W},$ we define $$\mathfrak{L}_{n}(\mathcal{W}):=\bigvee \{T(\xi_1)T(\xi_2)\cdots T(\xi_n)h \: : \: \xi_i\in E , h \in \mathcal{W} \},$$ for $n\in\mathbb N$ and $\mathfrak{L}_{0}(\mathcal{W}):=\mathcal{W}$. Then $\mathcal W$ is called {\it wandering} (cf.  \cite{H16}) for $(\sigma , T)$, if  $\mathcal{W}$ is orthogonal to  $\mathfrak{L}_{n}(\mathcal{W}),$ for all $n \in \mathbb{N}.$ We say that $(\sigma, T)$ has {\it generating wandering subspace property} if
$$\mathcal{H}=\bigvee_{n\geq0}\mathfrak{L}_n(\mathcal{W}),~\mbox{where $\mathcal{W}=\mbox{ker}\wt{T}^*.$}$$
\end{itemize}

\end{definition}


\begin{theorem}{\rm ({\bf Muhly and Solel})}\label{ThmW}
Let $(\sigma , T)$ be an isometric, covariant representation of $E$  on a Hilbert space $\mathcal{H}$. Then the representation $(\sigma, T)$ decomposes into a direct sum $(\sigma_1, T_1)\bigoplus(\sigma_2, T_2)$ on $\mathcal{H}=\mathcal{H}_1 \bigoplus \mathcal{H}_2$ where $(\sigma_1, T_1)=(\sigma, T)|_{\mathcal{H}_1}$ is an induced representation and $(\sigma_2, T_2)=(\sigma, T)|_{\mathcal{H}_2}$ is fully coisometric. The above decomposition is unique in the sense that if $\mathcal{K}$  reduces $(\sigma, T),$  and if the restriction $(\sigma, T)|_{\mathcal{K}}$ is induced (resp. fully coisometric), then $\mathcal{K} \subseteq \mathcal{H}_1$(resp. $\mathcal{K} \subseteq \mathcal{H}_2$).  Moreover, $\mathcal{H}_1:=\bigoplus_{n \in \mathbb N_0}\mathfrak{L}_{n}(\mathcal{W}),$ and hence	
	$$\mathcal{H}_2:=\left(\bigoplus_{n \in \mathbb N_0}\mathfrak{L}_{n}(\mathcal{W})\right)^{\bot}=\bigcap_{n\in  \mathbb N_0}\mathfrak{L}_{n}(\mathcal{H}),$$ where $\mathcal{W}:=\mbox{ran}(I_{\mathcal{H}}-\widetilde{T}\widetilde{T}^*)$ and $T_0:=\sigma.$
\end{theorem}

\section{Halmos-Richter-type Wandering subspace theorem for covariant representations of $C^*$-correspondences}\label{Sec2}

\begin{definition}	
	 Let $(\sigma, T)$ be a completely bounded, covariant representation of $E$ on a Hilbert space $\mathcal{H}.$ If the inequality
\begin{align}\label{W1}
\| \wt{T}_2(\eta \otimes h)\|^2+\|\eta \otimes h\|^2 \leq 2 \|(I_E \otimes  \wt{T})(\eta \otimes h)\|^2
\end{align}
 is satisfied for each $\eta \in E^{\otimes 2}, h \in \mathcal{H};$ then we say that the covariant representation $(\sigma, T)$ is \it{concave}.
\end{definition}

\begin{lemma}\label{L1}
Let $(\sigma, T)$ be a concave completely bounded, covariant representation of $E$ on a Hilbert space $\mathcal{H}.$ Then
\begin{enumerate}
\item[(1)] $\|\wt{T}(\xi \otimes h)\| \geq \|\xi \otimes h\|, \: ~\mbox{for each}~\xi \in E, h \in \mathcal{H} ;$
\item[(2)] $\|\wt{T}_n(\zeta \otimes h)\|^2 \leq \| \zeta \otimes h\|^2 + n(\|(I_{E^{\otimes (n-1)}} \otimes \wt{T})(\zeta \otimes h)\|^2-\|\zeta \otimes h\|^2), ~\mbox{for each}~$ $\zeta \in E^{\otimes n}, n \in \mathbb{N}.$
\end{enumerate}
\end{lemma}
\begin{proof}
First we shall prove the conclusion (2) by Mathematical induction. Note that Inequality (\ref{W1})  is equivalent to the operator inequality
\begin{align}\label{W2}
\wt{T}^*_2\wt{T}_2-I_{E^{\otimes 2} \otimes \mathcal{H}} \leq  2((I_E \otimes \wt{T}^*\wt{T})-I_{E^{\otimes 2} \otimes \mathcal{H}}).
\end{align}
Assume  for fixed $n\in \mathbb N$
\begin{align}\label{W3}
\wt{T}^*_n\wt{T}_n-I_{E^{\otimes n} \otimes \mathcal{H}} \leq n((I_{E^{\otimes n-1}} \otimes \wt{T}^*\wt{T})-I_{E^{\otimes n} \otimes \mathcal{H}}).
\end{align}
Then the inequality
\begin{align*}
\wt{T}^*_{n+1}\wt{T}_{n+1}-I_{E^{\otimes n+1} \otimes \mathcal{H}} \leq (n+1)((I_{E^{\otimes n}} \otimes \wt{T}^*\wt{T})-I_{E^{\otimes n+1} \otimes \mathcal{H}}),
\end{align*}
follows from
\begin{align*}
&(I_{E^{\otimes n}} \otimes \wt{T}^*)\wt{T}^*_n\wt{T}_n(I_{E^{\otimes n}} \otimes \wt{T})-(I_{E^{\otimes n}} \otimes \wt{T}^*\wt{T})\\
& \leq  n((I_{E^{\otimes n}} \otimes \wt{T}^*)(I_{E^{\otimes n-1}} \otimes \wt{T}^*\wt{T})(I_{E^{\otimes n}} \otimes \wt{T})-(I_{E^{\otimes n}} \otimes \wt{T}^*\wt{T}))\\
&=n((I_{E^{\otimes n-1}} \otimes \wt{T}_2^*\wt{T}_2)-(I_{E^{\otimes n}} \otimes \wt{T}^*\wt{T}))\\
& \leq n(2(I_{E^{\otimes n}} \otimes \wt{T}^*\wt{T})-I_{E^{\otimes n+1 }\otimes \mathcal{H}}-(I_{E^{\otimes n}} \otimes \wt{T}^*\wt{T})),
\end{align*}
where the last inequality follows from Inequality (\ref{W2}). Hence the conclusion (2) holds.
It follows from (2) that
\begin{align*}
\|\zeta \otimes h\|^2+n(\|(I_{E^{\otimes n-1}} \otimes \wt{T})(\zeta \otimes h)\|^2-\|\zeta \otimes h\|^2) \geq 0, 
\end{align*}
for each $\zeta \in E^{\otimes n}, h \in \mathcal{H}, n \in \mathbb{N}.$ Then $$\|(I_{E^{\otimes n-1}} \otimes \wt{T})(\zeta \otimes h)\|^2\geq \frac{n-1}{n} \|\zeta \otimes h\|^2. $$ Using the properties of the creation operators it is evident that 
 $$\|\wt{T}(\xi \otimes h)\| \geq \frac{n-1}{n}\|\xi \otimes h\|~\mbox{for all}~n \in \mathbb{N}.$$ Letting $n \rightarrow \infty$ we obtain  $\|\xi \otimes h\|\leq \|\wt{T}(\xi \otimes h)\|, \: \xi \in E.$
\end{proof}
\begin{definition}
Let $(\sigma, T)$ be a completely bounded, covariant representation of $E$ on a Hilbert space $\mathcal{H}.$ We say that $(\sigma,T)$ is {\it analytic} if $$\bigcap_{n \geq 1}\wt{T}_n(E^{\otimes n} \otimes \mathcal{H})=\{0\}.$$
\end{definition}
  Let $\mathcal{K}$  be a  $(\sigma, T)$-invariant subspace  of $\mathcal{H}$. Then $\mathcal{K} \ominus \wt{T}(E \otimes \mathcal{K})=\mbox{ker} \wt{T}^*|_{\mathcal{K}}$ is a wandering subspace for $(\sigma, T).$ Conversely, let $\mathcal{W}$ be a wandering subspace for $(\sigma, T)$. Then  the subspace $\mathcal{K} :=\bigvee_{n\geq 0}\mathfrak{L}_n({\mathcal{W}})$ is $(\sigma, T)$-invariant  and  in this case $\mathcal{W}=\mathcal{K} \ominus \wt{T}(E \otimes \mathcal{K}).$ Indeed,
  \begin{align*}\mathcal{K} \ominus \wt{T}(E \otimes \mathcal{K})&=\bigvee_{n\geq 0}\mathfrak{L}_n({\mathcal{W}}) \ominus \wt{T}(E \otimes \bigvee_{n\geq 0}\mathfrak{L}_n({\mathcal{W}}))\\&=\bigvee_{n\geq 0}\mathfrak{L}_n({\mathcal{W}}) \ominus \bigvee_{n\geq 1}\mathfrak{L}_n({\mathcal{W}})=\mathcal{W}.
\end{align*}
Hence the wandering subspace  $\mathcal{W}$ is uniquely determined by the invariant subspace $\mathcal{K}$. This combined with the following theorem, which is a generalization of Richter's wandering subspace theorem for analytic operators  \cite[Theorem 1]{R88}, gives us one-to-one correspondence between the set of all wandering subspaces of $\mathcal{H}$ to the set of all $(\sigma, T)$-invariant subspaces of $\mathcal{H}.$

\begin{theorem}\label{T1}
Let $(\sigma, T)$ be an analytic, concave, completely bounded, covariant representation of $E$ on a Hilbert space $\mathcal{H}.$ If $\mathcal{K}$ is a $(\sigma, T)$-invariant subspace of $\mathcal{H}$, then there exists a wandering subspace $\mathcal{W}$ for $(\sigma, T)$ such that
\begin{align*}
\mathcal{K}= \bigvee_{n \in \mathbb{N}_0}\mathfrak{L}_n(\mathcal{W}).
\end{align*}
In particular,  $(\sigma, T)$ has {\rm generating wandering subspace property}, that is, $$\mathcal{H}=\bigvee_{n \in \mathbb{N}_0}\mathfrak{L}_n(\mathcal{W}),~\mbox{and}~\mathcal{W}=\mbox{ker}\wT^*.$$
\end{theorem}
\begin{proof}
Let $\mathcal{W}=\mathcal{K} \ominus \wt{T}(E \otimes \mathcal{K})(=\mathcal{K} \cap \mbox{ker}\wt{T}^*) $. Then clearly $\mathcal{W}$ is a wandering subspace for $(\sigma, T)$, and $\mathcal{W} \subseteq \mathcal{K}.$ Since $\mathcal{K}$ is $(\sigma, T)$-invariant,
\begin{align*}
\bigvee_{n \in \mathbb{N}_0}\mathfrak{L}_n(\mathcal{W}) \subseteq \mathcal{K}.
\end{align*}

Now we shall prove the reverse inclusion. Define $(\rho, V):=(\sigma, T)|_{\mathcal{K}}$. Then the covariant representation $(\rho, V)$ is analytic and satisfy Inequality  (\ref{W1}).
Apply Lemma \ref{L1} (1)  to $(\rho, V),$ then $\wt{V}^*\wt{V}-I_{E \otimes \mathcal{K}}$ is a positive operator. Let $\mathcal{D}_V:=(\wt{V}^*\wt{V}-I_{E \otimes \mathcal{K}})^{1/2}.$  Then
\begin{align*}
\|\mathcal{D}_V(\xi \otimes h)\|^2=\|\wt{V}(\xi \otimes h)\|^2-\|\xi \otimes h\|^2,
\end{align*}
where $\xi \in E, h \in \mathcal{K}.$ Using Lemma \ref{L1} (2) we get
\begin{align}\label{Eq4}
\nonumber &\|\wt{V}_n(\zeta \otimes h)\|^2-\|\zeta \otimes h\|^2 \\ \nonumber\leq & n\left(\langle(I_{E^{\otimes n-1}} \otimes \wt{V}^*\wt{V})(\zeta \otimes h)-\zeta \otimes h,\zeta \otimes h\rangle\right) \\
=& n\|(I_{E^{\otimes n-1}} \otimes \mathcal{D}_V)(\zeta \otimes h)\|^2, \:\: \zeta \in E^{\otimes n}, h \in \mathcal{K.}
\end{align}
Since $\wt{V}$ is bounded below, it is easy to check that $\wt{V}^*\wt{V}$ is invertible and hence $L:=(\wt{V}^*\wt{V})^{-1}\wt{V}^*$ is a left inverse of $\wt{V}.$  Let  $Q:=\wt{V}L$ and $P:=I_{\mathcal{K}}-Q$. Then $P$ and $Q$ are orthogonal projections with $\mbox{ran} P= \mbox{ker} Q=\mbox{ker} \wt{V}^*=\mathcal{W}.$
For each $ n \in \mathbb{N}$, we define $L^n: \mathcal{K} \rightarrow E^{\otimes n} \otimes \mathcal{K}$  by
\begin{align}
L^n:=(I_{E^{\otimes n-1}} \otimes L)(I_{E^{\otimes n-2}} \otimes L)  \cdots (I_{E} \otimes L)L.
\end{align}
Fix $h \in \mathcal{K}$, we have
\begin{align*}
(I_{\mathcal{K}}-\wt{V}_nL^n)h &=\sum_{j=0}^{n-1}(\wt{V}_jL^j-\wt{V}_{j+1}L^{j+1})h\\&=\sum_{j=0}^{n-1}\wt{V}_j(I_{E^{\otimes j} \otimes \mathcal{K}}-(I_{E^{\otimes j}} \otimes \wt{V}L))L^jh\\
&=\sum_{j=0}^{n-1}\wt{V}_j(I_{E^{\otimes j}} \otimes P)L^{j}h \in \bigvee_{n \in \mathbb{N}_0}\mathfrak{L}_n(\mathcal{W}).
\end{align*}
Since $\bigvee_{n \in \mathbb{N}_0}\mathfrak{L}_n(\mathcal{W})$  is weakly closed, it is enough to prove that the sequence $\{(I_{\mathcal{K}}-\wt{V}_nL^n)h\}$ has a weakly convergent subsequence and it converges  to $h$(weakly).
 First using Mathematical induction we shall show that  for each  $n \in \mathbb{N}$,
\begin{align}\label{claim}
\|h\|^2=\sum_{j=0}^{n-1}\|(I_{E^{\otimes j}} \otimes P)L^jh\|^2+\|L^nh\|^2+\sum_{j=1}^{n}\|(I_{E^{\otimes j-1}} \otimes \mathcal{D}_V)L^jh\|.
\end{align}
Note that the $n=1$ case follows from
\begin{align*}
\|h\|^2&=\|Ph\|^2+\|Qh\|^2=\|Ph\|^2+\|Lh\|^2+\|\wt{V}Lh\|^2-\|Lh\|^2\\
&=\|Ph\|^2+\|Lh\|^2+\|\mathcal{D}_VLh\|^2.
\end{align*}
Assume $n \geq 1,$ it follows that
\begin{align}\label{claim1}
 \nonumber &\|L^nh\|^2\\ \nonumber =&\|(I_{E^{\otimes n}} \otimes P)L^nh\|^2+\|L^{n+1}h\|^2+\|(I_{E^{\otimes n}}\otimes \wt{V}L)L^nh\|^2-\|L^{n+1}h\|^2 \\
=&\|(I_{E^{\otimes n}} \otimes P)L^nh\|^2+\|L^{n+1}h\|^2+\|(I_{E^{\otimes n}} \otimes \mathcal{D}_V)L^{n+1}h\|^2.
\end{align}
Using  the previous equation and the induction hypothesis for $n \in \mathbb{N},$ we see that
\begin{align*}
\|h\|^2&=\sum_{j=0}^{n-1}\|(I_{E^{\otimes j}} \otimes P)L^jh\|^2+\|L^nh\|^2+\sum_{j=1}^n\|(I_{E^{\otimes j-1}} \otimes \mathcal{D}_V)L^jh\|^2\\
&=\sum_{j=0}^n\|(I_{E^{\otimes j}} \otimes P)L^jh\|^2+\|L^{n+1}h\|^2+\sum_{j=1}^{n+1}\|(I_{E^{\otimes j-1}} \otimes \mathcal{D}_V)L^jh\|^2.
\end{align*}
Hence the claim has been proved for each $n\in\mathbb N$.

Let $l,m \in \mathbb{N}$ with  $l \leq m,$  then
\begin{align*}
&(inf\{\|\wt{V}_nL^nh\|^2-  \|L^nh\|^2, \: l \leq n \leq m\})\sum_{n=l}^m \frac{1}{n} \\
 \leq &\sum_{n=l}^m \frac{1}{n}(\|\wt{V}_nL^nh\|^2-  \|L^nh\|^2) \\
 \leq &\sum_{n=l}^m \|(I_{E^{\otimes n-1}} \otimes \mathcal{D}_V)L^nh\|^2 \:\:(\mbox{by Inequality \eqref{Eq4})}\\
\leq & \|h\|^2 \:\: (\mbox{by Equation}\eqref{claim}).
\end{align*}
  Using Equation \eqref{claim1} we get that $\{\|L^nh\|^2\}$ is a non-increasing sequence of non-negative numbers.  Hence $\lim_{n \rightarrow \infty}\|L^nh\|^2$ exists and
\begin{align}\label{claim2}
\liminf_{n \rightarrow \infty}\|\wt{V}_nL^nh\|^2=\lim_{n \rightarrow \infty}\|L^nh\|^2.
\end{align}
Then there exists a weakly convergent subsequence of $\{\wt{V}_nL^nh\}$, say $\{\wt{V}_{n_j}L^{n_j}h\}$ which converges to $ w$(weakly) for some $w \in \mathcal{K}$. Let  $N \in \mathbb{N}$  be such that $n_j \geq N$, implies $\wt{V}_{n_j}L^{n_j}h \in \wt{V}_{n_j}(E^{\otimes n_j} \otimes \mathcal{K}) \subseteq \wt{V}_{N}(E^{\otimes N} \otimes \mathcal{K}).$  Since $(\sigma, T)$ is analytic and  $\wt{V}_N(E^{\otimes N} \otimes \mathcal{K})$ is weakly closed (by Lemma \ref{L1}(a)), it follows that $w=0.$ Hence $(I_{\mathcal K}-\wt{V}_{n_j}L^{n_j})h \rightarrow h$(weakly).
\end{proof}

\begin{remark}\label{rm1}
 We shall see that  $(I_{\mathcal K}-\wt{V}_{n_j}L^{n_j})h \rightarrow h$ in norm. This will follow, if $\|L^nh\| \rightarrow 0.$
 Let $h \in \mathcal{K},$  by the previous theorem  there exists a sequence $h_n$ converging to $h$ in the norm  such that $$h_n=\sum_{m=0}^{N_n}\wt{V}_{m}(\eta_{m,n} \otimes h_{m,n}), \:\: \:\mbox{for all}~ \eta_{m,n} \in E^{\otimes m}, h_{m,n} \in \mathcal{W}, N_n \in \mathbb{N}_0.$$
Let $\varepsilon >0.$ Then there exists $N \in \mathbb{N}$ such that $\|h-h_N\| < \varepsilon.$ For each $n > N,$ it follows that
$$\|L^nh\|=\|L^n(h-h_N)\| \leq \|h - h_N\| < \varepsilon. $$
Hence $\|L^nh\| \rightarrow 0.$
\end{remark}
   Using the previous remark and Equation \eqref{claim} we see that for every $h\in\mathcal K$ we have
   \begin{align}
\|h\|^2=\sum_{j=0}^{\infty}\|(I_{E^{\otimes j}} \otimes P)L^jh\|^2+\sum_{j=1}^{\infty}\|(I_{E^{\otimes j-1}} \otimes \mathcal{D}_V)L^jh\|.
\end{align}

\section{Cauchy Dual and Shimorin-type decomposition for covariant representations of $C^*$-correspondences}\label{Sec3}

 The  main theorem of this section, Theorem \ref{MT1}, establishes  two sufficient conditions for the  covariant representation $(\sigma, T)$ for the existence of the following Wold-type decomposition:

\begin{definition}
Let $(\sigma, T)$ be a completely bounded, covariant representation of $E$ on a Hilbert space $\mathcal{H}.$ We say that the covariant representation $(\sigma, T)$ admits {\it Wold-type decomposition} if there exists a wandering subspace $\mathcal{W}$
for $(\sigma, T)$ which decomposes $\mathcal{H}$  into the following direct sum of two $(\sigma, T)$-reducing subspaces  
\begin{align}
\mathcal{H}=\bigvee_{n\geq0}\mathfrak{L}_n(\mathcal{W}) \bigoplus \bigcap_{n \geq 1}\wt{T}_n(E^{\otimes n} \otimes \mathcal{H})
\end{align}
such that the restriction of $(\sigma,T)$ on the subspace $\bigcap_{n \geq 1}\wt{T}_n(E^{\otimes n} \otimes \mathcal{H})$  is isometric as well as fully co-isometric covariant representation. 
\end{definition}
\begin{remark}
Since the wandering subspace  is  unique,  $\mathcal{W}=\mbox{ker}\wt{T}^*$. Throughout this section we use notations $\mathcal{H}_{\infty}$ and $\mathcal{W}$ for

$$\mathcal{H}_{\infty}:=\bigcap_{n \geq 1}\wt{T}_n(E^{\otimes n} \otimes \mathcal{H}); \:\: \mathcal{W}:=\mbox{ker} \wt{T}^*=\mathcal{H}\ominus \wt{T}(E \otimes \mathcal{H}). $$
\end{remark}

Let $(\sigma, T)$ be a completely bounded, covariant representation of $E$ on a Hilbert space $\mathcal{H},$ such that $\wt{T}$ is left invertible.
  Suppose that $\wt{T}$ is invertible, then $\mathcal{H}_{\infty}=\mathcal{H}$ and $ \mathcal{W}=\{0\}.$ If $\wt{T}$ is not invertible, then the subspace  $\mathcal{H}_{\infty}$ is $(\sigma, T)$-invariant and $\mathcal{W}$ is a non-trivial wandering subspace for $(\sigma, T).$
  Since $\wt{T}$ is left invertible, $\wt{T}^*\wt{T}$ is invertible, define the operator $L: \mathcal{H} \rightarrow E \otimes \mathcal{H}$ by $$ L:=(\wt{T}^*\wt{T})^{-1}\wt{T}^*.$$ 
It is easy to see that  $\mbox{ker}L=\mathcal{W}$  and  $P_{\mathcal{W}}=I_{\mathcal{H}}-\wt{T}L,$  where $P_{\mathcal{W}}$  denote the orthogonal projection on $\mathcal{W}.$  We denote this orthogonal projecton on $\mathcal{W}$ by $P$(instead of  $P_{\mathcal{W}}).$
   Let $ n \in \mathbb{N}$, define $L^n: \mathcal{H} \rightarrow E^{\otimes n} \otimes \mathcal{H}$  by
\begin{align}
L^n:=(I_{E^{\otimes n-1}} \otimes L)(I_{E^{\otimes n-2}} \otimes L)  \cdots (I_{E} \otimes L)L.
\end{align}
Note that 
 \begin{align}
 (I_{\mathcal{H}}-\wt{T}_nL^n)h=\sum_{j=0}^{n-1}\wt{T}_j(I_{E^{\otimes j}} \otimes P)L^jh, \:\: h \in \mathcal{H}.
 \end{align}
 
\begin{lemma}\label{L2}
$\mbox{ker} L^n=\bigvee\{\wt{T}_j(\eta_j \otimes w): \: \eta_j \in E^{\otimes j}, w \in \mathcal{W}, j=0,1, \dots ,n-1\}.$
\end{lemma}
\begin{proof}
 If $L^nw=0,$ then $$w=w-\wt{T}_nL^nw=\sum_{j=0}^{n-1}\wt{T}_j(I_{E^{\otimes j}} \otimes P)L^jw,$$  and it shows the inclusion $\subseteq.$ Since $L\wt{T}=I_{\mathcal H},$ we have $\mbox{ker} L=\mbox{ker} \wt{T}^*=\mathcal{W},$ and the reverse inclusion follows.
\end{proof}

Let $(\sigma, T)$ be a concave completely bounded, covariant representation of $E$ on  $\mathcal{H}.$ Define the operator $U: \mathcal{H} \longrightarrow \mathcal{F}(E) \otimes \mathcal{W}$ by $$Uh:=\sum_{n=0}^{\infty}(I_{E^{\otimes n}} \otimes P)L^nh,$$
where $P=I_{\mathcal{H}}-Q, \: Q=\wt{T}L.$  Then the operator $U$ is well defined contraction. Indeed,
\begin{align*}
\|Uh\|^2=& \sum_{n=0}^{\infty}\langle L^{n*}(I_{E^{\otimes n}} \otimes (I_{\mathcal H}-\wt{T}L))L^nh, h\rangle\\
=&\sum_{n=0}^{\infty}\langle(L^{n*}L^n- L^{n*}(I_{E^{\otimes n}} \otimes \wt{T}(\wt{T}^*\wt{T})^{-1}\wt{T}^*\wt{T})L^{n+1})h,h \rangle \\
=& 
\sum_{n=0}^{\infty}\langle(L^{n*}L^n- L^{(n+1)*}(I_{E^{\otimes n}} \otimes \wt{T}^*\wt{T})L^{n+1})h,h \rangle \\
\leq &\sum_{n=0}^{\infty} \langle(L^{n*}L^n- L^{(n+1)*}L^{n+1})h,h \rangle~ (because ~ \wt{T}^*\wt{T} \geq I_{E\otimes\mathcal H}) \\
=& \|h\|^2- \lim_{n \rightarrow \infty}\langle L^{n*}L^nh, h \rangle \leq \|h\|^2  (\mbox{by Equation}~\eqref{claim1}).
\end{align*}
\begin{lemma} \label{L32}
$\mbox{ker}  U=\mathcal{H}_{\infty}.$
\end{lemma}
\begin{proof}
 Suppose that for all $n \in \mathbb{N},$ $(I_{E^{\otimes n}} \otimes P)L^nh=0.$  It follows that  $$h=h-\wt{T}_nL^nh+\wt{T}_nL^nh=\sum_{j=0}^{n-1}\wt{T}_j(I_{E^{\otimes j}} \otimes P)L^jh+\wt{T}_nL^nh=\wt{T}_nL^nh,$$ 
and we get $h \in \mathcal{H}_{\infty}.$  Conversely, assume that $h \in \mathcal{H}_{\infty},$ then 
$$ h=\wt{T}_{n+1}(\eta_{n+1} \otimes h_{n+1})~\mbox{ for some }~\eta_{n+1} \in E^{\otimes n+1}, h_{n+1} \in \mathcal{H}, n \in \mathbb{N}.$$ Hence
\begin{align*}
(I_{E^{\otimes n}} \otimes P)L^nh=&(I_{E^{\otimes n}} \otimes P)L^n\wt{T}_{n+1}(\eta_{n+1} \otimes h_{n+1})\\=&(I_{E^{\otimes n}} \otimes P\wt{T})(\eta_{n+1} \otimes h_{n+1})=0.\qedhere
\end{align*}
\end{proof}
By Remark \ref{rm1}, if $(\sigma, T)$ is analytic then  for all $h \in \mathcal{H},$ we have $L^n h\rightarrow 0$ as $n \rightarrow \infty.$
Therefore, $U$ is unitary if and only if $(\sigma, T)$ is analytic.

Let $(\sigma, T)$ be a completely bounded, covariant representation of $E$ on a Hilbert space $\mathcal{H},$ such that $\wt{T}$ is left invertible.
Define $\wt{T}': E \otimes \mathcal{H} \longrightarrow \mathcal{H}$ by  $$\wt{T}':=\wt{T}(\wt{T}^*\wt{T})^{-1}.$$ 
Since $\wt{T}^*\wt{T}(\phi(a) \otimes I_{\mathcal{H}})=\wt{T}^*\sigma(a)\wt{T}=(\phi(a) \otimes I_{\mathcal{H}})\wt{T}\wt{T}^*$ for each $a \in \mathcal{M},$  it is easy to observe that  $(\wt{T}^*\wt{T})^{-1}(\phi(a) \otimes I_{\mathcal{H}})=(\phi(a) \otimes I_{\mathcal{H}})(\wt{T}^*\wt{T})^{-1}$ and hence $$\wt{T}'(\phi(a) \otimes I_{\mathcal{H}})=\wt{T}(\phi(a) \otimes I_{\mathcal{H}})(\wt{T}^*\wt{T})^{-1}=\sigma(a)\wt{T}'.$$
Therefore the corresponding covariant representation $(\sigma, T')$  of $E$ on the Hilbert space $\mathcal{H}$ is completely bounded by Lemma \ref{MSL}. 
\begin{definition}
We say that  the covariant representation $(\sigma, T')$ defined above is {\it Cauchy dual} of $(\sigma, T).$
\end{definition}
\begin{notation}
$\mathcal{H}'_{\infty}:=\bigcap_{n \geq 1}\wt{T}'_n(E^{\otimes n} \otimes \mathcal{H}) \:\:\mbox{and}\:\:  \mathcal{W}':=\mbox{ker} \wt{T}'^*. $
\end{notation}
Obviously, $\mathcal{W'}=\mbox{ker}\wt{T}^*=\mathcal{W}.$ 
\begin{proposition}\label{CD}
(1)$(\mathcal{H}'_{\infty})^{\bot}=\bigvee_{n \geq 0}\mathfrak{L}_n(\mathcal{W});$ \\
(2)$\mathcal{H}_{\infty}^{\: \bot}=\bigvee_{n \geq 0}\mathfrak{L}_n'(\mathcal{W}),$ where $$\mathfrak{L}_n'(\mathcal{W}):=\bigvee \{T'(\xi_1)T'(\xi_2)\cdots T'(\xi_n)h \: : \: \xi_i\in E , ~1\leq i\leq n, ~h \in \mathcal{W} \}.$$
\end{proposition}
\begin{proof}
First we shall prove {\it(1)}. Note that
\begin{align*}
(\mathcal{H}'_{\infty})^{\bot} =& \left(\bigcap_{n \geq 0}L^{*n}(E^{\otimes n} \otimes \mathcal{H})\right)^{ \bot}=\bigvee_{n \geq0}\big(L^{*n}(E^{\otimes n} \otimes \mathcal{H})\big)^{\bot}\\
=&\bigvee_{n \geq0}\mbox{ker}L^{n}=\bigvee_{n \geq0}\mathfrak{L}_n(\mathcal{W}).
\end{align*}
The last equality follows from Lemma \ref{L2}. Note that
\begin{align}\label{eqn1}
(\wt{T'})'=\wt{T'}(\wt{T'}^*\wt{T'})^{-1}=\wt{T},
\end{align}
that is, the Cauchy dual of  $(\sigma, T')$  is $(\sigma, T).$ Hence {\it(2)} follows.
\end{proof}

\begin{corollary}
\begin{enumerate}
\item[(1)] $(\sigma, T)$ is analytic if and only if $(\sigma, T')$  has the generating wandering subspace property;
\item[(2)] $(\sigma, T)$ has the generating wandering subspace property if and only if $(\sigma, T')$ is analytic.
\end{enumerate}
\end{corollary}

\begin{corollary}\label{C35}
$(\sigma, T)$ admits Wold-type decomposition if and only if $(\sigma, T')$ does. Indeed, $\mathcal{H}_{\infty}=\mathcal{H}_{\infty}'$ and $\bigvee_{n \geq 0}\mathfrak{L}_n(\mathcal{W})=\bigvee_{n \geq 0}\mathfrak{L}'_n(\mathcal{W}).$
\end{corollary}

\begin{proposition}
If $\mathcal{H}_{\infty}'$ is $(\sigma, T')$-reducing, then $\mathcal{H}_{\infty}' \subseteq \mathcal{H}_{\infty}.$
\end{proposition}
\begin{proof}
Suppose $\mathcal{H}_{\infty}'$ is  $(\sigma, T')$-reducing. From Equation \eqref{eqn1}, it follows that $\wt{T}=\wt{T'}(\wt{T'}^*\wt{T'})^{-1}$. Thus $\mathcal{H}_{\infty}'$ is also  $(\sigma, T)$-reducing. Note that  $ \wt{T}|_{E \otimes \mathcal{H}_{\infty}'}=\wt{T'}_{\infty}(\wt{T'}^*_{\infty}\wt{T'}_{\infty})^{-1},$ where $\wt{T'}_{\infty}:=\wt{T'}|_{E \otimes \mathcal{H}_{\infty}'}.$ Therefore, $\mathcal{H}_{\infty}'$ is  $(\sigma, T)|_{\mathcal{H}_{\infty}'}$-reducing.  Since $\wt{T'}_{\infty}$ is invertible, we conclude that $\wt{T}|_{E \otimes \mathcal{H}_{\infty}'}$ is also invertible, hence $\mathcal{H}_{\infty}'=\wt{T}(E \otimes \mathcal{H}_{\infty}') \subseteq \mathcal{H}_{\infty},$ which proves the  proposition.
\end{proof}
\begin{corollary}
Let $(\sigma, T)$ be analytic and $\mathcal{H}_{\infty}'$ be $(\sigma, T')$-reducing. Then $(\sigma, T)$  has the generating wandering subspace property.
\end{corollary}


The following theorem gives a sufficient condition for the wandering subspace property for an analytic covariant representation $(\sigma, T).$

\begin{theorem}\label{TC}
Let $(\sigma, T)$ be a concave, completely bounded, covariant representation of $E$ on  $\mathcal{H}.$
 Then $\mathcal{H}_{\infty}$ is  $(\sigma, T)$-reducing and $(\sigma, T)|_{\mathcal{H}_{\infty}}$ is an isometric and fully co-isometric, covariant  representation of $E$ on the Hilbert space $\mathcal{H}_{\infty}.$
\end{theorem}
\begin{proof}
Let $ h \in \mathcal{H}_{\infty},$ then by Lemma \ref{L32}, $h=\wt{T}_nL^nh, n \in \mathbb{N}.$ Since  $(\sigma, T)$ is concave, by Equation \eqref{claim1} $\|L^{n+1}h\| \leq \|L^nh\| \leq \|h\|.$ Also, by Equation \eqref{claim2}, we obtain $\lim_{n \rightarrow  \infty}\|L^nh\|=\|h\|,$ therefore $\|L^nh\|=\|h\|,$ for each $ n \geq 0.$ In particular, $\|Lh\|=\|h\|,$ that is, $L$ is an isometry. For every $\xi \in E,$ $$ \|\xi \otimes h\|=\|L\wt{T}(\xi \otimes h)\|=\|\wt{T}(\xi \otimes h )\|.$$ Hence $\wt{T}|_{E \otimes \mathcal{H}_{\infty}}$ is an isometry.

Since $\wt{T}: E \otimes \mathcal{H}_{\infty} \longrightarrow \mathcal{H}_{\infty}$ is onto, there exists $\kappa \in E \otimes \mathcal{H}_{\infty}$ such that $\wt{T}(\kappa)=h.$ Then $  \wt{T}^*(h)=\wt{T}^*\wt{T}(\kappa)=\kappa \in E \otimes \mathcal{H}_{\infty}$, and it follows that  $\mathcal{H}_{\infty}$ is a reducing subspace for $(\sigma, T).$  Hence $(\sigma, T)|_{\mathcal{H}_{\infty}}$ is an isometric and fully co-isometric covariant representation.
\end{proof}

The following result is a generalization of Shimorin's Wold-type decomposition \cite[Theorem 3.6]{S01} and Theorem \ref {ThmW}:

\begin{theorem}\label{MT1}
Let $(\sigma, T)$ be a completely bounded, covariant representation of $E$ on  $\mathcal{H},$ which satisfies any one of the following conditions:
\begin{enumerate}
\item[(1)]  $(\sigma, T)$ is concave, that is, \\ $\| \wt{T}_2(\eta \otimes h)\|^2+\|\eta \otimes h\|^2 \leq 2 \|(I_E \otimes  \wt{T})(\eta \otimes h)\|^2, \eta \in E^{\otimes 2}, h \in \mathcal{H};$
\item[(2)]  for any  $\zeta \in E^{\otimes 2} \otimes \mathcal{H},\kappa \in E \otimes \mathcal{H} $
\begin{align}\label{Eq12}
\|(I_{E} \otimes \wt{T})(\zeta)+\kappa\|^2 \leq 2(\|\zeta\|^2+\|\wt{T}(\kappa)\|^2).
\end{align}
\end{enumerate}
Then $(\sigma, T)$ admits Wold-type decomposition. In particular, if $(\sigma, T)$ is analytic, then $\mathcal{W}=\mathcal{H} \ominus \wt{T}(E \otimes \mathcal{H})$ is the generating wandering subspace for $(\sigma, T),$ that is, $\mathcal{H}=\bigvee_{n \in \mathbb{N}_0}\mathfrak{L}_n(\mathcal{W}).$
\end{theorem}
\begin{proof}
First assume that $(\sigma, T)$ is concave. By Theorem \ref{TC}, $(\sigma, T)|_{\mathcal{H}_{\infty}}$ is an isometric and fully co-isometric covariant representation.  It remains to prove  the relation $\mathcal{H} \ominus \mathcal{H}_{\infty}=\bigvee_{n \geq 0}\mathfrak{L}_n(\mathcal{W}).$ Since $\mathcal{H} \ominus \mathcal{H}_{\infty}$ is reducing for $(\sigma, T)$, then it is easy to check that  $(\sigma, T)|_{\mathcal{H} \ominus\mathcal{H}_{\infty}}$ is analytic, and therefore $\mathcal{W}=\mbox{ker}(\wt{T}|_{E \otimes (\mathcal{H} \ominus \mathcal{H}_{\infty})})^*.$
 Now apply Theorem \ref{T1} for the $(\sigma, T)$-invariant subspace $\mathcal{H} \ominus \mathcal{H}_{\infty},$  then $\mathcal{W}$ is a generating wandering subspace for $(\sigma, T)|_{\mathcal{H} \ominus \mathcal{H}_{\infty}},$ and from the uniqueness of the wandering subspaces we get the desired relation.

Next, assume that $(\sigma, T)$ satisfies the condition $(2)$ of the statement. Taking $\zeta=0 $ in the Inequality \eqref{Eq12},  we get $\wt{T}$ is left invertible. Let $\kappa \in E \otimes \mathcal{H} .$ Then   $\kappa=(\wt{T}^*\wt{T})^{-1}\xi$ for some $ \xi \in E \otimes \mathcal{H}$, and substituting $\kappa$ in the same inequality gives
$$\|(I_{E} \otimes \wt{T})(\zeta)+(\wt{T}^*\wt{T})^{-1/2}\xi\|^2 \leq 2(\|\zeta\|^2+\|\xi\|^2). $$
Define the  operator $X: (E^{\otimes 2} \otimes \mathcal{H}) \oplus (E \otimes \mathcal{H}) \longrightarrow E \otimes \mathcal{H}$ by
$$X(\xi_1,\xi_2):=(I_{E} \otimes \wt{T})\xi_1 +(\wt{T}^*\wt{T})^{-1/2}\xi_2.$$ Then from the previous inequality, $\|X\|\leq \sqrt{2},$ which gives $XX^* \leq 2 I_{E \otimes \mathcal{H}},$ and we obtain  $$I_E \otimes \wt{T}\wt{T}^*+(\wt{T}^*\wt{T})^{-1} \leq 2I_{E \otimes \mathcal{H}}.$$
 Now, substitute $\wt{T}=\wt{T'}(\wt{T'}^*\wt{T'})^{-1}$ in the previous inequality and use Equation \eqref{eqn1}, we have
$$I_{E} \otimes \wt{T'}(\wt{T'}^*\wt{T'})^{-2}\wt{T'}^*+(\wt{T'}^*\wt{T'}) \leq 2I_{E \otimes \mathcal{H}}.$$
Multiply this inequality by $I_{E} \otimes \wt{T'}^*$ on the left side and by $I_E \otimes \wt{T'}$ on the right, we obtain $$I_{E^{\otimes 2} \otimes \mathcal{H}}+\wt{T'}^*_2\wt{T'}_2\leq 2 (I_{E} \otimes \wt{T'}^*\wt{T'}) .$$  This shows that the Cauchy dual $(\sigma, T')$ of $(\sigma, T)$ is concave. Therefore  $(\sigma, T')$ admits Wold-type decomposition. Hence, using Corollary \ref{C35}, we get the desired conclusion.
\end{proof}
\begin{corollary}
Let $(\sigma, T)$ be an analytic completely bounded, covariant representation of $E$ on a Hilbert space $\mathcal{H}$ which satisfies the Inequality  \eqref{Eq12}. Then $(\sigma, T)$  has the generating  wandering subspace property.
\end{corollary}
\begin{corollary}
Let $(\sigma, T)$ be a completely bounded, covariant representation of $E$ on  $\mathcal{H}$ which satisfies
\begin{align}\label{Eq13}
\|\wt{T}(\xi)\|^2+\|\wt{T}^*_2\wt{T}(\xi)\|^2 \leq 2\|\wt{T}^*\wt{T}(\xi )\|^2, \:  \xi \in E \otimes \mathcal{H},
 \end{align} and let $\wt{T}$ be bounded below. Then $(\sigma, T)$ admits Wold-type decomposition. In particular, if $(\sigma, T)$ is analytic, then $(\sigma, T)$  has  the generating wandering subspace property.
\end{corollary}

\begin{proof}
We have that $$\langle \wt{T}^*\wt{T}(\xi), \xi  \rangle+\langle \wt{T}^*\wt{T}_2\wt{T}_2^*\wt{T}^*(\xi), \xi  \rangle \leq 2\langle \wt{T}^*\wt{T}(\xi ), \wt{T}^*\wt{T}(\xi )\rangle,$$ where $\xi \in E \otimes \mathcal{H}.$ Since $\wt{T}$ is left invertible,  $\xi =(\wt{T}^*\wt{T})^{-1}(\eta)$ for some $\eta \in E \otimes \mathcal{H},$ substituting  it  in  the above  inequality, we get
 $$\langle (\wt{T}^*\wt{T})^{-1}(\eta), \eta \rangle+\langle (I_E \otimes \wt{T}^*\wt{T})(\eta), \eta \rangle \leq 2 \langle \eta, \eta \rangle, $$ that is, \begin{align} \label{Eq14}
 I_E \otimes \wt{T}\wt{T}^*+(\wt{T}^*\wt{T})^{-1} \leq 2I_{E \otimes \mathcal{H}}.
 \end{align}
 In fact,  Equation (\ref{Eq14}) is equivalent to Equation (\ref{Eq13}).
 
Define the  operator $X: (E^{\otimes 2} \otimes \mathcal{H}) \oplus (E \otimes \mathcal{H}) \longrightarrow E \otimes \mathcal{H}$  as in the  previous theorem. Inequality  \eqref{Eq14} yields   $$\|(I_{E} \otimes \wt{T})(\xi_1)+(\wt{T}^*\wt{T})^{-1/2}\xi_2\|^2=\|X(\xi_1, \xi_2)\|^2 \leq 2(\|\xi_1\|^2+\|\xi_2\|^2). $$
Further, substituting $\eta=(\wt{T}^*\wt{T})^{-1/2}\xi_2$ we have
$$\|(I_{E} \otimes \wt{T})(\xi_1)+\eta\|^2 \leq 2(\|\xi_1\|^2+\|\wt{T}(\eta)\|^2).$$ This shows that Inequality \eqref{Eq14} is equivalent to Inequality \eqref{Eq12}. Hence from Theorem \ref{MT1} $(\sigma, T)$ admits Wold-type decomposition.
\end{proof}

\section{Generating wandering subspaces for doubly commuting covariant representations of product systems}\label{Sec4}
Wandering subspaces for covariant representations of subproduct systems of $C^*$-correspondences is studied in \cite{SHV16}. In this section we recall several notions related to product systems of $C^*$-correspondences from (see \cite{F02,S06, S08, SZ08}) and explore wandering subspaces for doubly commuting covariant representations of product systems. 

\begin{definition}
\begin{itemize}
\item[(1)]  Let $k\in \mathbb N.$ A {\it product system} (cf. \cite{F02})
$\be$ is defined as a family of $C^*$-correspondences $\{E_1, \ldots,E_k\},$ along with the unitary
isomorphisms $t_{i,j}: E_i \ot E_j \to E_j \ot E_i$ ($i>j$). Using these isomorphisms, for all
${\bf n}=(n_1, \cdots, n_k) \in \Nk$ the correspondence $\be ({\bf n})$ is identified with $E_1^{\ot^{ n_1}} \ot \cdots \ot E_k^{\ot^{n_k}}.$ We define maps $t_{i,i} := \id_{E_i \ot E_i}$ and $t_{i,j} := t_{j,i}^{-1}$ when $i<j.$

\item[(2)]  Assume $\be$ to be a product system over $\Nk$. A {\it completely bounded, covariant representation} (cf. \cite{S08}) of $\be$ on a
	Hilbert space $\mathcal H$ is defined as a tuple $(\sigma, T^{(1)}, \ldots, T^{(k)})$, where $\sigma$ is a
	representation of the $C^*$-algebra $\mathcal M$ on $\mathcal H$,  $T^{(i)}:E_i \to B(\mathcal H)$ are  completely bounded linear  maps satisfying
	\[ T^{(i)}(a \xi_i b) = \sigma(a) T^{(i)}(\xi_i) \sigma(b), \;\;\; a,b \in \mathcal M, \xi_i \in E_i,\]
	and satisfy the commutation relation
	\begin{equation} \label{rep} \wT^{(i)} (I_{E_i} \ot \wT^{(j)}) = \wT^{(j)} (I_{E_j} \ot \wT^{(i)}) (t_{i,j} \ot I_{\mathcal H})\end{equation}
	with $1\leq i,j\leq k$.
	Moreover, the completely bounded, covariant representation $(\sigma, T^{(1)}, \ldots, T^{(k)})$ is called {\it isometric} if each $(\sigma, T^{(i)})$ is isometric as a covariant representation of the $C^*$-correspondence $E_i$, and similarly the notion of {\it fully co-isometric} is defined.
\end{itemize}
\end{definition}

For each $1\leq i \leq k$
and $l \in \bn$ define $\wT^{(i)}_l: E_i^{\ot l}\otimes \mathcal H\to \mathcal H$ by
\[ \wT^{(i)}_l (\xi_1 \ot \cdots \ot \xi_l \ot h) := T^{(i)} (\xi_1) \cdots T^{(i)}(\xi_l) h\]
 where $\xi_1, \ldots, \xi_l \in E_{i}, h \in \mathcal H$. Then
\begin{equation}\label{eqnn}\wT^{(i)}_l=\wT^{(i)}(I_{E_i} \otimes \wT^{(i)}) \cdots (I_{E^{\otimes l-1}_i} \otimes  \wT^{(i)}).\end{equation}
Similarly for $\mathbf{n}=(n_1, \cdots, n_k) \in \mathbb{N}_0^k $, we use notation $\wT_{\mathbf{n}}:\mathbb{E}(\mathbf{n})\otimes \mathcal{H} \longrightarrow \mathcal{H}$ for
 $$\wT_{\mathbf{n}}:=\wT^{(1)}_{n_1}\left(I_{E_1^{\otimes n_1}} \otimes\wT^{(2)}_{n_2}\right) \cdots \left(I_{E_1^{\otimes n_1} \otimes \cdots \otimes E_{k-1}^{\otimes {n_{k-1}}}} \otimes\wT^{(k)}_{n_k}\right).$$
 Let us define the linear map $T_{\mathbf{n}}: \mathbb{E}(\mathbf{n}) \longrightarrow B(\mathcal{H})$  (cf. \cite{S08}) by $$T_{\mathbf{n}}(\xi)h:=\wT_{\mathbf{n}}(\xi \otimes h), ~ \xi \in \mathbb{E}(\mathbf{n}), h \in \mathcal{H}.$$
We use  $I_k$ for $\{1,2, \dots ,k\}.$ Let $\alpha=\{\alpha_1, \dots ,\alpha_n \} \subseteq I_k$,   define $$\mathbb{N}_0^{\alpha}:=\{\mathbf{m}=(m_{\alpha_1}, \cdots ,m_{\alpha_n})\: : \: m_{\alpha_j} \in \mathbb{N}_0, \:1 \leq j \leq n\}.$$ 
For each $\mathbf{m}=(m_{\alpha_1}, \cdots ,m_{\alpha_n})\in \mathbb{N}_0^{\alpha}$, the map $\wT_{\mathbf{m}}^{\alpha}:\mathbb{E}(\mathbf{m})\otimes \mathcal{H} \longrightarrow \mathcal{H}$ is defined by
 $$\wT_{\mathbf{m}}^{\alpha}:=\wT^{(\alpha_1)}_{m_{\alpha_1}}\left(I_{E_{\alpha_1}^{\otimes m_{\alpha_1}}} \otimes\wT^{(\alpha_2)}_{m_{\alpha_2}}\right) \cdots \left(I_{E_{\alpha_1}^{\otimes m_{\alpha_1}} \otimes \cdots \otimes E_{\alpha_{n-1}}^{\otimes {m_{\alpha_{n-1}}}}} \otimes\wT^{(\alpha_n)}_{m_{\alpha_n}}\right).$$

\begin{definition}	
	Let  $\mathcal K$ be a closed subspace of a Hilbert space $ \mathcal H.$  The subspace $\mathcal K$ is called {\it invariant}({ resp. \it reducing}) (cf. \cite{SZ08}) for a covariant representation $(\sigma, T^{(1)}, \ldots, T^{(k)})$ on $\mathcal H,$ if   $ \mathcal{K}$ is $(\sigma, T^{(i)})$-invariant(resp.  $(\sigma, T^{(i)})$-reducing) subspace for  $1 \leq i \leq k.$  Then it is evident that the natural `restriction' of this representation to $\mathcal K$ provides a new representation
	of $\be$ on $\mathcal K$, which is called a {\it summand} of $(\sigma, T^{(1)}, \ldots, T^{(k)})$ and will be
	denoted by $(\sigma, T^{(1)}, \ldots, T^{(k)})|_{\mathcal K}$.	
	\end{definition}

 Moreover, for  a closed   $\sigma(\mathcal{M})$-invariant  subspace  $\mathcal K,$ we use symbol $$\mathfrak{L}_{\mathbf{m}}^{\alpha}(\mathcal{K}):=\bigvee \{T_{m_{\alpha_1}}^{(\alpha_1)}(\eta_{\alpha_1}) \cdots T_{m_{\alpha_n}}^{(\alpha_p)}(\eta_{\alpha_n})h\: : \: \:\: \eta_{\alpha_j} \in E_{\alpha_j}^{\otimes m_{\alpha_j}}, 1 \leq j \leq n , h \in \mathcal{K}\}.$$ Clearly $\mathfrak{L}_{\mathbf{m}}^{\alpha}(\mathcal{K})=\overline{\wT_{\mathbf{m}}^{\alpha}(\mathbb{E}(\mathbf{m}) \otimes \mathcal{K})}.$ We use $[\mathcal K]_{T_{\alpha}}$ for the smallest closed \\$(\sigma, T^{(\alpha_1)}, \dots, T^{(\alpha_n)})$-invariant subspace of $\mathcal{H}$ containing $\mathcal{K},$ that is,
 \begin{align}
 [\mathcal{K}]_{T_{\alpha}}:=\bigvee_{\mathbf{m} \in \mathbb{N}_0^{\alpha}}\mathfrak{L}_{\mathbf{m}}^{\alpha}(\mathcal{K}).
 \end{align}

Further, we simply write $[\mathcal K]_{T^{(i)}}$  if $\alpha$ is a singleton set $\{i\}.$
\begin{definition}\begin{enumerate}\item 
 We say that the $\sigma(\mathcal{M})$-invariant closed subspace $\mathcal{K}$ is said to be {\it wandering} for  the covariant representation $(\sigma, T^{(\alpha_1)} ,\dots, T^{(\alpha_n)})$ if 
$$\mathcal{K}\perp\mathfrak{L}_{\mathbf{m}}^{\alpha}(\mathcal{K})~\mbox{ for each }~\mathbf{m} \in \mathbb{N}_0^{\alpha} \setminus\{0\}.$$
\item
The covariant representation $(\sigma, T^{(\alpha_1)} ,\dots, T^{(\alpha_n)})$  is said to have the {\it generating wandering subspace  property} if there exists a wandering subspace $\mathcal{K} \subseteq \mathcal{H}$ for $(\sigma, T^{(\alpha_1)} ,\dots, T^{(\alpha_n)})$ such that $[\mathcal{K}]_{T_{\alpha}}=\mathcal{H},$ that is, $$\mathcal{H}=\bigvee_{\mathbf{m} \in \mathbb{N}_0^{\alpha}}\mathfrak{L}_{\mathbf{m}}^{\alpha}(\mathcal{K}).$$
\end{enumerate}
\end{definition}

\begin{notation}
Let $\alpha=\{\alpha_1, \dots ,\alpha_n \}$ be a non-empty subset of $I_k,$ define the closed subspace $\mathcal{W}_{\alpha}$ of $\mathcal{H}$ by
\begin{align}
\mathcal{W}_{\alpha}:=\bigcap_{i=1}^n (\mathcal{H} \ominus \wt{T}^{(\alpha_i)}(E_i \otimes \mathcal{H})).
\end{align}
Again, if $\alpha=\{i\}$ we simply write $\mathcal{W}_i:=\mathcal{H} \ominus \wt{T}^{(i)}(E_i \otimes \mathcal{H}).$ Therefore $$\mathcal{W}_{\alpha}=\bigcap_{\alpha_i\in\alpha}\mathcal{W}_{\alpha_i}.$$
\end{notation}

\begin{definition}  \label{dcom}
	A  completely bounded, covariant representation $(\sigma, T^{(1)}, \ldots,$ $ T^{(k)})$ of  $\be$ on a Hilbert space $\mathcal H$ is said to be {\it doubly
	commuting}  (cf. \cite{S08}) if for each distinct $i,j \in \{1,\ldots,k\}$ we have
	\begin{equation}\label{doubly}\wT^{(j)^*} \wT^{(i)} =
	(I_{E_j} \ot \wT^{(i)})  (t_{i,j} \ot I_{\mathcal H})  (I_{E_i} \ot \wT^{(j)^*}).
	\end{equation}
\end{definition}

For distinct $i, j \in \onek$, a simple calculation (cf. \cite[p. 460]{SZ08}) using Equation \eqref{doubly} yields 
\begin{align}\label{eqn*}
\wt{T}^{(i)^*}\wt{T}^{(i)}\wt{T}^{(j)^*}\wt{T}^{(j)} =
\wt{T}^{(j)^*}\wt{T}^{(j)}\wt{T}^{(i)^*}\wt{T}^{(i)} .
\end{align}
Thus the operators $\{I_{\mathcal H}-\wt{T}^{(i)^*}\wt{T}^{(i)}\}_{i=1}^k$ commute to each other.

The following proposition is essential in order to extend Theorem \ref{MT1} for the covariant representation $(\sigma, T^{(1)}, \dots, T^{(k)}).$
\begin{proposition}\label{P21}
Let $\mathbb{E}$ be a product system of $C^*$-correspondences over $\mathbb{N}_0^k$ and let $(\sigma, T^{(1)}, \dots, T^{(k)})$ be a doubly commuting completely bounded, covariant representation of  $\mathbb{E}$ on a Hilbert space $\mathcal{H}.$ Then for each non-empty subset $\alpha \subseteq I_{k}$, the subspace $\mathcal{W}_{\alpha}$ is $(\sigma, T^{(j)})$-reducing,  where $j \notin \alpha.$
\end{proposition}
\begin{proof}
Let $\alpha=\{\alpha_1, \dots,\alpha_n\}$ be a non-empty subset of $I_k$ and  let $j \notin \alpha.$ Since  $\mathcal{W}_l=\mbox{ker}\widetilde{T}^{(l)^*}$ for all $1 \leq l \leq k,$ we get $\mathcal{W}_{\alpha}=\bigcap_{i=1}^{n}\mbox{ker}\widetilde{T}^{(\alpha_i)^*}.$  Let $\eta_j \in E_j, \xi_{\alpha_i} \in E_{\alpha_i}, w_{\alpha} \in \mathcal{W}_{\alpha}$ and $h \in \mathcal{H}.$ By doubly commutativity of $(\sigma, T^{(1)}, \dots, T^{(k)}),$ we obtain
\begin{align*}
&\langle \widetilde{T}^{(j)} (\eta_j \otimes w_{\alpha}), \widetilde{T}^{(\alpha_i)}(\xi_{\alpha_i} \otimes h)\rangle \\=&\langle  (I_{E_{\alpha_i}} \otimes \widetilde{T}^{(j)})(t_{j, \alpha_i} \otimes I_{\mathcal{H}})(I_{E_j} \otimes \widetilde{T}^{(\alpha_i)^*})(\eta_j \otimes w_{\alpha}), \xi_{\alpha_i} \otimes h\rangle  =0.
\end{align*}
 It follows that $\widetilde{T}^{(j)}(E_j \otimes \mathcal{W}_{\alpha}) \bot \widetilde{T}^{(\alpha_i)}(E_{\alpha_i} \otimes \mathcal{H})$ and hence $\widetilde{T}^{(j)}(E_j \otimes \mathcal{W}_{\alpha}) \subseteq \mathcal{W}_{\alpha_i}$ for all $\alpha_i \in \alpha$ and $j \notin  \alpha.$ Therefore $\mathcal{W}_{\alpha}$ is $(\sigma, T^{(j)})$-invariant.  Also,
\begin{align*}
&\langle \widetilde{T}^{(j)^*}w_{\alpha}, (I_{E_j} \otimes \widetilde{T}^{(\alpha_i)})(\eta_j \otimes \xi_{\alpha_i} \otimes h)\rangle
\\ =& \: \langle w_{\alpha},  \widetilde{T}^{(\alpha_i)}(I_{E_{\alpha_i}} \otimes \widetilde{T}^{(j)})(t_{j, \alpha_i} \otimes I_{\mathcal{H}})(\eta_j \otimes \xi_{\alpha_i} \otimes h) \rangle \\
=& \: \langle\widetilde{T}^{(\alpha_i)^*}w_{\alpha}, (I_{E_{\alpha_i}} \otimes \widetilde{T}^{(j)})(t_{j, \alpha_i} \otimes I_{\mathcal{H}})(\eta_j \otimes \xi_{\alpha_i} \otimes h) \rangle=0,
\end{align*}
for all $\alpha_i \in \alpha$ and $h \in \mathcal{H}.$  Note that $\mbox{ker}(I_E \otimes \wt{T}^{(\alpha_i)^*})=E \otimes \mbox{ker}\wt{T}^{(\alpha_i)^*}$ (see  proof of \cite[ Lemma 1.6 ]{Sk09}), then  $\widetilde{T}^{(j)^*} w_{\alpha} \in ((I_{E_j} \otimes \widetilde{T}^{(\alpha_i)})(E_j \otimes E_{\alpha_i} \otimes \mathcal{H}))^{\bot}=\mbox{ker}(I_{E_j} \otimes \widetilde{T}^{(\alpha_i)^*})=E_j \otimes \mathcal{W}_{\alpha_i}$ and then $\widetilde{T}^{(j)^*}\mathcal{W}_{\alpha} \subseteq  E_j \otimes \mathcal{W}_{\alpha_i} $ for all $\alpha_i \in \alpha.$ It follows that $\widetilde{T}^{(j)^*}\mathcal{W}_{\alpha} \subseteq  E_j \otimes \mathcal{W}_{\alpha} $ for $j \notin \alpha.$  This completes the proof.
\end{proof}
We denote the cardinality of $\alpha$ by $\# \alpha.$
\begin{theorem}\label{T22}
Let $(\sigma, T^{(1)}, \dots, T^{(k)})$ be a doubly commuting completely bounded, covariant representation of the product system $\mathbb{E}$ on a Hilbert space $\mathcal{H}$ such that for any $(\sigma, T^{(i)})$-reducing subspace $\mathcal{K}_i$, the subspace $$\mathcal{K}_i \ominus \widetilde{T}^{(i)}(E_i \otimes \mathcal{K}_i)\left(=\mbox{ker}\widetilde{T}^{(i)^*}|_{\mathcal{K}_i}\right)$$ is a generating wandering subspace for $(\sigma, T^{(i)})|_{\mathcal{K}_i}~\mbox{where}~ i=1, \ldots ,k.$ Then for each non-empty set $\alpha=\{\alpha_1, \ldots , \alpha_n\} \subseteq I_k,$ the covariant representation $(\sigma, T^{(\alpha_1)}, \dots, T^{(\alpha_n)})$ has a generating wandering subspace property. Moreover, the corresponding generating  wandering subspace is given by $$\mathcal{W}_{\alpha}=\bigcap_{i=1}^{n}(\mathcal{H} \ominus \widetilde{T}^{(\alpha_i)}(E_{\alpha_i} \otimes \mathcal{H})).$$
\end{theorem}
\begin{proof}
First, we shall show  that $\mathcal{W}_{\alpha}=\bigcap_{i=1}^{n}\mathcal{W}_{\alpha_i},$ where $\mathcal{W}_{\alpha_i}=\mbox{ker}\widetilde{T}^{(\alpha_i)^*}$ is the wandering subspace for  the covariant representation $(\sigma, T^{(\alpha_1)}, \dots, T^{(\alpha_n)}).$  Let $w_1,w_2 \in \mathcal{W}_{\alpha} $ and $\mathbf{m}=(m_{\alpha_1}, \cdots ,m_{\alpha_n}) \in \mathbb{N}_0^{\alpha}\setminus\{0\}.$ Without loss of generality we can assume that $m_{\alpha_1} \neq 0.$  Then
\begin{align*}
\langle w_1, \widetilde{T}^{\alpha}_{\mathbf{m}}(\eta_{\alpha} \otimes w_2) \rangle &=\langle w_1, \widetilde{T}^{(\alpha_1)}_{m_1}(I_{E^{\otimes m_1}_{ \alpha_1}} \otimes \widetilde{T}^{\beta}_{\mathbf{m}-m_1 \mathbf{e_1}}) (\eta_{\alpha} \otimes w_2)\rangle \\
&=\langle \widetilde{T}^{(\alpha_1)^*}_{m_1}w_1,  (I_{E^{\otimes m_1}_{ \alpha_1}} \otimes \widetilde{T}^{\beta}_{\mathbf{m}-m_1\mathbf{e_1}}) (\eta_{\alpha} \otimes w_2) \rangle=0,
\end{align*}
where $\beta=\alpha \setminus \{ \alpha_1\},$ $
\eta_{\alpha} \in \mathbb E(\mathbf{m})$ and $\mathbf{e_i} \in \mathbb{N}_0^{\alpha}$ whose $\alpha_i^{\mbox{th}}$ entry is 1 and all other entries are $0.$ 
 Next,  we  need to show that $\mathcal{H}=\bigvee_{\mathbf{m} \in \mathbb{N}_0^{\alpha}}\mathfrak{L}^{\alpha}_{\mathbf{m}}(\mathcal{W}_{\alpha}).$
  Suppose that $\alpha=\{i\},$  then  by hypothesis   $\mathcal{W}_i$ is a generating wandering subspace for $(\sigma, T^{(i)}), i =1, \ldots, k.$ Now for $ \#\alpha \geq 2,$ it is enough to show that $[\mathcal{W}_{\alpha} ]_{T^{(\alpha_i)}}=\mathcal{W}_{\alpha \setminus \{\alpha_i\}}$ for any $\alpha_i \in \alpha.$ Because for $\alpha_i,\alpha_j \in \alpha,$ one can repeat the procedure  and  obtain
 \begin{align*}
 [\mathcal{W}_{\alpha} ]_{T_{\{\alpha_i, \alpha_j\}}} &= \bigvee_{\mathbf{m} \in \mathbb{N}^2_0}\mathfrak{L}^{\{\alpha_i, \alpha_j\}}_{\mathbf{m}}( \mathcal{W}_{\alpha}) = \bigvee_{m_2 \in \mathbb{N}_0}\mathfrak{L}^{(\alpha_i)}_{m_2}([\mathcal{W}_{\alpha}]_{T^{(\alpha_j)}}) \\
 &=[[\mathcal{W}_{\alpha}]_{T^{(\alpha_i)}}]_{T^{(\alpha_j)}}=\mathcal{W}_{\alpha \setminus \{\alpha_i, \alpha_j\}},
 \end{align*}
 and continue this procedure until the set $\alpha \setminus \{\alpha_i, \alpha_j\}$ becomes a singleton set and  apply the assumption for the singleton set.

 To this end, let $\alpha \subseteq I_k, \:  \#\alpha \geq 2$ and $\alpha_i \in \alpha.$ Consider the subspace $F=\mathcal{W}_{\alpha \setminus \{\alpha_i\}} \ominus \widetilde{T}^{(\alpha_i)}(E_{\alpha_i} \otimes \mathcal{W}_{\alpha \setminus \{\alpha_i\}}).$ Now by Proposition \ref{P21}, $\mathcal{W}_{\alpha \setminus \{\alpha_i\}}$ is $(\sigma, T^{(\alpha_i)})$-reducing and therefore  $F=\mathcal{W}_{\alpha \setminus \{\alpha_i\}} \cap \mathcal{W}_{\alpha_i}=\mathcal{W}_{\alpha}.$ On the other hand, since $\mathcal{W}_{\alpha \setminus \{\alpha_i\}}$ is $(\sigma, T^{(\alpha_i)})$-reducing, by hypothesis $F=\mathcal{W}_{\alpha}$ is the wandering subspace for $(\sigma, T^{(\alpha_i)}).$ Hence
  $$ [\mathcal{W}_{\alpha} ]_{T^{(\alpha_i)}}=\mathcal{W}_{\alpha \setminus \{\alpha_i\}}.$$ This completes the proof.
\end{proof}
\begin{corollary}\label{C23}
Let $(\sigma, T^{(1)}, \dots, T^{(k)})$ be a doubly commuting completely bounded, covariant representation of the product system $\mathbb{E}$ on a Hilbert space $\mathcal{H}$ such that  $(\sigma, T^{(1)}, \dots, $ $T^{(k)})$ is analytic  and satisfies one of the following properties:
\begin{enumerate}
\item[(1)] $(\sigma, T^{(i)})$ is concave for each $i=1, \dots ,k,$
\item[(2)]   $\|(I_{E_i} \otimes \wt{T}^{(i)})(\zeta)+\xi\|^2 \leq 2(\|\zeta\|^2+\|\wt{T}^{(i)}(\xi)\|^2),$ for each $\zeta \in E_i^{\otimes 2} \otimes \mathcal{H},\xi \in E_i \otimes \mathcal{H}, i= 1, \dots ,k.$
\end{enumerate}
Then for any non-empty set $\alpha=\{\alpha_1, \dots , \alpha_k\} \subseteq I_k,  \: \mathcal{W}_{\alpha}$ is a generating wandering subspace for $(\sigma, T^{(\alpha_1)}, \dots, T^{(\alpha_k)}).$
\end{corollary}
\begin{proof}
Let $\mathcal{K}_i$ be a $(\sigma , T^{(i)})$-reducing subspace of $\mathcal{H}$ for $1\leq i \leq k.$ Then  the covariant representation $(\sigma, T^{(i)})|_{\mathcal{K}_i}$ also satisfies one of the conditions (1)  and  (2), respectively. Thus by Theorem \ref{MT1},  $\mathcal{K}_i \ominus \wt{T}^{(i)}(E \otimes \mathcal{K}_i)$ is a generating wandering subspace for $(\sigma, T^{(i)})|_{\mathcal{K}_i}.$ Hence from Theorem \ref{T22} the proof follows.
\end{proof}

The following main theorem of the paper is a generalization of \cite[Theorem 2.5]{CDSS14}:

\begin{theorem}\label{T24}
	Let $(\sigma, T^{(1)}, \dots, T^{(k)})$ be a  completely bounded, covariant representation of the product system $\mathbb{E}$ on a Hilbert space $\mathcal{H}$ such that\\
	$\|(I_{E_i} \otimes \wt{T}^{(i)})(\zeta)+\xi\|^2 \leq 2(\|\zeta\|^2+\|\wt{T}^{(i)}(\xi)\|^2), \:  \zeta \in E_i^{\otimes 2} \otimes \mathcal{H},\xi \in E_i \otimes \mathcal{H}$ \\or $(\sigma, T^{(i)})$ is concave for all $i= 1, \dots ,k.$ Then
	\begin{enumerate}
		\item $(\sigma, T^{(1)}, \dots, T^{(k)})$  is doubly commuting, and
		\item $(\sigma, T^{(i)})$ is analytic for all $i= 1, \dots ,k.$
	\end{enumerate} if and only if
	\begin{enumerate}
		\item[(a)] For any non-empty set $\alpha=\{\alpha_1, \dots , \alpha_n\} \subseteq I_k,  \: \mathcal{W}_{\alpha}$ is a generating wandering subspace for $(\sigma, T^{(\alpha_1)}, \dots, T^{(\alpha_n)}),$  and for $n \geq 2 ,\:  [\mathcal{W}_{\alpha} ]_{{T}^{(\alpha_i)}}=\mathcal{W}_{\alpha \setminus \{\alpha_i\}},$ for all $\alpha_i \in \alpha,$
		\item[(b)] $(I_{E_j} \otimes \wt{T}^{(i)})(t_{i ,j} \otimes I_{\mathcal{H}})(I_{E_i} \otimes \wt{T}^{(j)^*}\wt{T}^{(j)})=\wt{T}^{(j)^*}\wt{T}^{(j)}(I_{E_j} \otimes \wt{T}^{(i)})(t_{i, j} \otimes I_{\mathcal{H}})$ for all $1 \leq i<j\leq k.$
	\end{enumerate}
\end{theorem}
\begin{proof}
	The necessary part follows from Theorem \ref{T22}. Indeed,   for $1 \leq i < j \leq k,$   by the  assumption of the doubly commutativity,  we have
	\begin{align*}
	(I_{E_j} \otimes \wt{T}^{(i)})(t_{i, j} \otimes I_{\mathcal{H}})(I_{E_i} \otimes \wt{T}^{(j)^*}\wt{T}^{(j)}) &= \wt{T}^{(j)^*}\wt{T}^{(i)}(I_{E_i} \otimes\wt{T}^{(j)}) \\
	\* &=\wt{T}^{(j)^*}\wt{T}^{(j)}(I_{E_j} \otimes \wt{T}^{(i)})(t_{i, j} \otimes I_{\mathcal{H}}).
	\end{align*}
	
	Conversely, suppose that (a) and (b) hold. First we shall show that $(\sigma, T^{(1)}, \dots,$ $ T^{(k)})$  is doubly commuting. Let $1 \leq i \leq k$ be fixed. By assumption(a), $\mathcal{W}_i=[\mathcal{W}_{\{i,j\}}]_{T^{(j)}}$ for all $1 \leq i \neq j \leq k.$  It follows  that $\mathcal{W}_i$ is $(\sigma, T^{(j)})$-invariant subspace for all $1 \leq j \neq i\leq k.$ Let $h \in \mathcal{H},$ since $[\mathcal{W}_i]_{T^{(i)}}= \mathcal{H},$ thus there exists a sequence $h_n$ converging to $h$ such that $$h_n=\sum_{m=0}^{N_n}\wt{T}^{(i)}_{m}(\eta_{m,n} \otimes h_{m,n}), \:\: \: \eta_{m,n} \in E_i^{\otimes m}, h_{m,n} \in \mathcal{W}_i \: \mbox{and } \: N_n \in \mathbb{N}_0.$$
	Then for any $1 \leq j \neq i \leq k \: \mbox{and}\: \eta_j \in E_j,$  we obtain
	\begin{align*}\wt{T}^{(i)^*}\wt{T}^{(j)}(\eta_j \otimes h_{n})&=\sum_{m=0}^{N_n}\wt{T}^{(i)^*}\wt{T}^{(j)}(\eta_j \otimes \wt{T}^{(i)}_{m}(\eta_{m,n} \otimes h_{m,n}))\\&=\sum_{m=1}^{N_n}\wt{T}^{(i)^*}\wt{T}^{(j)}(\eta_j \otimes \wt{T}^{i}_{m}(\eta_{m,n} \otimes h_{m,n})),\end{align*}
	since $\mathcal{W}_i=\mbox{ker}\wt{T}^{(i)^*}.$ On the other hand, using notation $\eta \otimes h$ for $\eta_{m,n} \otimes h_{m,n}$ we have
	\begin{align*}
	&(I_{E_i} \otimes \wt{T}^{(j)})(t_{j , i} \otimes I_{\mathcal{H}})(I_{E_j} \otimes \wt{T}^{(i)^*})(\eta_j \otimes h_{n})
	\\
	=& \sum_{m=1}^{N_n}(I_{E_i} \otimes \wt{T}^{(j)})(t_{j,i} \otimes I_{\mathcal{H}})(I_{E_j} \otimes \wt{T}^{(i)^*})(\eta_j \otimes \wt{T}^{(i)}_{m}(\eta \otimes h))   \\ = &\sum_{m=1}^{N_n}(I_{E_i} \otimes \wt{T}^{(j)})(t_{j,i} \otimes I_{\mathcal{H}})(I_{E_j} \otimes \wt{T}^{(i)^*}\wt{T}^{(i)})(I_{E_j \otimes E_i} \otimes \wt{T}^{(i)}_{m-1})(\eta_j \otimes \eta \otimes h)  \\
	=&\sum_{m=1}^{N_n}\wt{T}^{(i)^*}\wt{T}^{(j)}(I_{E_j} \otimes \wt{T}^{(i)})(I_{E_j \otimes E_i} \otimes \wt{T}^{(i)}_{m-1})(\eta_j \otimes \eta \otimes h) \\
	=&\sum_{m=1}^{N_n}\wt{T}^{(i)^*}\wt{T}^{(j)}(\eta_j \otimes \wt{T}^{i}_{m}(\eta \otimes h)),
	\end{align*}
	where the $3^{rd}$ equality follows from (b). It implies $$\wt{T}^{(i)^*}\wt{T}^{(j)}(\eta_j \otimes h_{m,n})= (I_{E_i} \otimes \wt{T}^{(j)})(t_{j,i} \otimes I_{\mathcal{H}})(I_{E_j} \otimes \wt{T}^{(i)^*})(\eta_j \otimes h_{m,n})$$ and then  by taking limit  we obtain $$\wt{T}^{(i)^*}\wt{T}^{(j)}(\eta_j \otimes h)=(I_{E_i} \otimes \wt{T}^{(j)})(t_{j,i} \otimes I_{\mathcal{H}})(I_{E_j} \otimes \wt{T}^{(i)^*})(\eta_j \otimes h).$$ Thus we have (1).
	
	Finally, from  Theorem \ref{MT1} we have that $$\mathcal{H}=[\mathcal{W}_i]_{T^{(i)}} \: \bigoplus \:  \bigcap_{m \geq 1} \wt{T}_m^{(i)}(E_i^{\otimes m} \otimes \mathcal{H}), $$  for all $1 \leq i \leq k.$ But by part (a), $[\mathcal{W}_i]_{T^{(i)}}=\mathcal{H}$ for all $1 \leq i \leq k.$ Thus $\bigcap_{m \geq 1} \wt{T}_m^{(i)}(E_i^{\otimes m} \otimes \mathcal{H})$=\{0\} for all $1 \leq i \leq k.$ 
\end{proof}


\subsection*{Acknowledgment}
Shankar V. is grateful to The LNM Institute of Information Technology for providing research facility and warm hospitality during a visit in March 2019. Shankar V. is supported by CSIR Fellowship (File No: 09/115(0782)/2017-EMR-I).


\begin{thebibliography}{1}
\bibitem{ARS96}
Aleman, A.; Richter S.; Sundberg C.; \textit{Beurling’s theorem for the Bergman space}, Acta Math.
\textbf{177} (1996), 275--310.

\bibitem{A89}
Arveson, William; \textit{Continuous analogues of {F}ock space}, Mem. Amer. Math. Soc. \textbf{80} (1989), no.~409, iv+66.

\bibitem{B49}
 Beurling, A.; \textit{On two problems concerning linear transformations in Hilbert space}, Acta Math.
\textbf{81} (1949), 239--255.

\bibitem{CDSS14}
Chattopadhyay, A.; Das, B. Krishna; Sarkar, Jaydeb; Sarkar, S.; \textit{Wandering subspaces of the Bergman space and the Dirichlet space over $\mathbb D^n$}, Integral Equations Operator Theory
\textbf{79} (2014), 567--577.

\bibitem{C77}
Cuntz, Joachim; \textit{Simple $C^*$-algebras generated by isometries}, Comm. Math. Phys.
\textbf{57}, (1977), no.~2, 173--185.


\bibitem{D11}
Douglas, Ronald G.; \textit{Variations on a theme of {B}eurling}, New York Journal of Mathematics
\textbf{17A}, (2011), 1--10.

\bibitem{F02}
Fowler, Neal J.; \textit{Discrete product systems of Hilbert bimodules}, Pacific J. Math.
\textbf{204} (2002), no.~2, 335--375.

\bibitem{H61}
Halmos, P. R.; \textit{Shifts on Hilbert spaces}, J. Reine Angew. Math.
\textbf{208} (1961), 102--112.

\bibitem{H16}
Helmer, Leonid; \textit{Generalized {I}nner-{O}uter {F}actorizations in {N}on
  {C}ommutative {H}ardy {A}lgebras}, Integral Equations Operator Theory
  \textbf{84} (2016), no.~4, 555--575.



\bibitem{I10}
Izuchi, Kei Ji; Izuchi, Kou Hei and Izuchi, Yuko; \textit{Wandering subspaces and the {B}eurling type {T}heorem {I}}, Arch. Math. (Basel)
  \textbf{95} (2010), no.~5, 439--446.







\bibitem{La95}
 Lance E.~C.; \textit{Hilbert {$C^*$}-modules}, London Mathematical Society
Lecture Note Series, vol. 210, Cambridge University Press, Cambridge, 1995, A
toolkit for operator algebraists.


\bibitem{M88}
Mandrekar, V.; \textit{ The validity of Beurling theorems in polydiscs},
  Proc. Amer. Math. Soc. \textbf{103} (1988), 145--148.


\bibitem{MR1648483}
Muhly, Paul~S.  and Solel, Baruch; \textit{Tensor algebras over
  {$C^*$}-correspondences: representations, dilations, and {$C^*$}-envelopes},
  J. Funct. Anal. \textbf{158} (1998), no.~2, 389--457.

\bibitem{MS99}
Muhly, Paul~S.  and Solel, Baruch; \textit{Tensor Algebras, {I}nduced Representations,
    and the {W}old Decomposition}, Canad. J. Math \textbf{51} (1999), no.~4, 850--880.

\bibitem{NF70}
Sz.-Nagy, B\'{e}la; Foia\c{s}, Ciprian; \textit{Harmonic analysis of operators on Hilbert space}, North-Holland Publishing Co., Amsterdam-London; American
   Elsevier Publishing Co., Inc., New York; Akad\'{e}miai Kiad\'{o}, Budapest, 1970, xiii+389 pp.



\bibitem{N92}
Nica, A.; \textit{$C^*$-algebras generated by isometries and Wiener-Hopf operators},
J. Operator Theory \textbf{27} (1992), no. 1, 17--52.


\bibitem{MR0355613}
Paschke, William~L.; \textit{Inner product modules over {$B^{\ast} $}-algebras},
Trans. Amer. Math. Soc. \textbf{182} (1973), 443--468.



\bibitem{P97}
Pimsner, Michael V.; \textit{ A class of {$C^*$}-algebras generalizing both
              {C}untz-{K}rieger algebras and crossed products by {${\bf
              Z}$}},
Free probability theory (Waterloo, ON, 1995), 189–212, Fields Inst. Commun., 12, Amer. Math. Soc., Providence, RI, 1997.


  \bibitem{Po89}
 Popescu,  Gelu; \textit{Isometric dilations for infinite sequences of noncommuting
   operators}, Trans. Amer. Math. Soc. \textbf{319} (1989), no.~2, 523-536.


  \bibitem{RT10}
Redett, D.; Tung, J.; \textit{Invariant subspaces in Bergman space over the bidisc}, Proc. Amer Math Soc \textbf{138} (2010),  2425--2430.

\bibitem{R88}
Richter, Stefan; \textit{Invariant subspaces of the Dirichlet shift},
J. reine angew. Math. \textbf{386} (1988), 205--220.

\bibitem{R74}
Rieffel, Marc A., \textit{Induced representations of $C^* $-algebras},
Advances in Math. \textbf{13} (1974), 176-257.

\bibitem{R69}
 Rudin, W.; \textit{Function theory in polydiscs}, W. A. Benjamin, Inc., New York-Amsterdam, 1969, vii+188 pp.


\bibitem{S13}
Sarkar, Jaydeb; Sasane, Amol; Wick, Brett D.; \textit{ Doubly commuting submodules of the Hardy module over polydiscs.},
Studia Math.  \textbf{217} (2013),  no. 2, 179–192. 

\bibitem{S14}
Sarkar, Jaydeb; \textit{ Wold decomposition for doubly commuting isometries},
Linear Algebra Appl. \textbf{445} (2014),289--301.

\bibitem{SHV16}
Sarkar, Jaydeb; Trivedi, Harsh; Veerabathiran, Shankar; \textit{Covariant representations of subproduct systems: Invariant subspaces and curvature}, New York J. Math.,  \textbf{24} (2018), 211-232.
	

\bibitem{S01}
Sergei, Shimorin; \textit{Wold-type decompositions and wandering subspaces for operators close to isometries
},
J. reine angew. Math. \textbf{531} (2001),  147--189.


\bibitem{SZ08}
Skalski, Adam; Zacharias, Joachim; \textit{Wold decomposition for representations of product systems of $C^*$-correspondences},
International J. Math. \textbf{19} (2008), no. 4, 455--479.

\bibitem{Sk09}
Skalski, Adam; \textit{On isometric dilations of product systems of $C^*$-correspondences and applications to families of contractions associated to higher-rank graphs.},
Indiana Univ. Math. J. \textbf{58} (2009), no. 5, 2227--2252.


\bibitem{Sl80}
S{\l}oci{\'n}ski, Marek; \textit{On the Wold-type decomposition of a pair of commuting isometries},
Ann. Polon. Math. \textbf{37} (1980), no. 3, 255--262.



\bibitem{S06}
Solel, Baruch; \textit{ Representations of product systems over semigroups and dilations of commuting CP maps},
  J. Funct. Anal. \textbf{180} (2006), no. 2, 593-618.

\bibitem{S08}
Solel, Baruch; \textit{Regular dilations of representations of product systems},
Math. Proc. R. Ir. Acad. \textbf{180} (2008), no. 1, 89--110.

\bibitem{TV19}
Trivedi, Harsh; Veerabathiran, Shankar; \textit{Wold decomposition for Doubly commuting isometric covariant representations of product systems}, preprint, arXiv:1903.07867,  (2019), 16pages.


\bibitem{V11}
Viselter, Ami; \textit{Covariant representations of subproduct systems}, Proc.
  Lond. Math. Soc. (3) \textbf{102} (2011), no.~4, 767--800. 



\bibitem{vN29}
v. Neumann, J.; \textit{Zur Algebra der Funktionaloperationen und Theorie der normalen
   Operatoren}, Math. Ann. \textbf{102} (1930), no.~1, 370--427. 

		


\bibitem{W}
Wold, Herman; \textit{A study in the analysis of stationary time series},
Almquist and Wiksell, Uppsala, 1938.



\end{thebibliography}
\end{document}